\documentclass[a4paper,12pt]{article} 
\usepackage{graphicx}
\usepackage[english]{babel} 
\usepackage{amsmath,amssymb,amsthm,amsfonts} 
\usepackage{latexsym}
\usepackage[all]{xy}
\usepackage{url}
\usepackage{mathrsfs}
\usepackage[latin1]{inputenc}

\usepackage{hyperref}
\hypersetup{
 pdftitle={A representative of R Gamma(N,T) for higher dimensional twists of Zpr(1)},
 pdfauthor={Alessandro Cobbe}}

\newcommand{\Q}{\mathbb Q}
\newcommand{\N}{\mathbb N}
\newcommand{\Z}{\mathbb Z}

\newcommand{\p}{\mathfrak p}
\newcommand{\Gal}{\mathrm{Gal}}
\renewcommand{\epsilon}{\varepsilon}
\newcommand{\Kur}{K_{\mathrm{nr}}}

\newcommand{\Zpnr}{\mathbb Z_p^\mathrm{nr}}
\newcommand{\calM}{\mathcal M}

\newcommand{\chinr}{\chi^{\mathrm{nr}}}
\newcommand{\rhonr}{\rho^{\mathrm{nr}}}
\newcommand{\rhoQpnr}{\rho^{\mathrm{nr}}_{\Q_p}}
\newcommand{\BdR}{B_\mathrm{dR}}
\newcommand{\Ucris}{U_\mathrm{cris}}
\newcommand{\Ind}{\mathrm{Ind}}
\newcommand{\coker}{\mathrm{coker}}

\newcommand{\CEP}{C_{EP}^{na}(N/K, V)}
\newcommand{\lra}{\longrightarrow}
\newcommand{\calF}{\mathcal F}

\newcommand{\calI}{\mathcal{I}}

\newcommand{\oo}{\mathcal{O}}

\newcommand{\Qpc}{\Q_p^c}
\newcommand{\Map}{\mathrm{Map}}
\newcommand{\Nnr}{\calI_{N/K}(\rhonr)}
\newcommand{\Lnr}{\calI_{L/K}(\rhonr)}
\newcommand{\Mnr}{\calI_{M/K}(\rhonr)}

\newcommand{\fFL}{f_{\calF,L}}

\newcommand{\tfFL}{\tilde{f}_{\calF,L}}
\newcommand{\calG}{\mathcal G}
\newcommand{\zz}{\mathcal Z}
\newcommand{\Gl}{\mathrm{Gl}}
\newcommand{\Gm}{\mathbb{G}_m}
\newcommand{\Qpnr}{{\mathbb Q_p^\mathrm{nr}}}

\newtheorem{thm_intro}{Theorem}

\newtheorem{thm}{Theorem}[section]
\newtheorem{lemma}[thm]{Lemma}
\newtheorem{coroll}[thm]{Corollary}
\newtheorem{prop}[thm]{Proposition}

\theoremstyle{remark}

\newtheorem{example}[thm]{Example}

\theoremstyle{definition}

\title{A representative of $R\Gamma(N,T)$\\
for higher dimensional twists of $\Z_p^r(1)$}

\author{Alessandro Cobbe}
\date{}
\begin{document}
\maketitle
\begin{abstract}
Let $N/K$ be a Galois extension of $p$-adic number fields and let $V$ be a de Rham representation of the absolute Galois group $G_K$ of $K$. In the case $V=\Q_p(1)$, the equivariant local $\epsilon$-constant conjecture describes the compatibility of the equivariant Tamagawa number conjecture with the functional equation of Artin $L$-functions and it can be formulated as the vanishing of a certain element $R_{N/K}$ in $K_0(\Z_p[G],\Qpc[G])$; a similar approach can be followed also in the case of unramified twists $\Q_p(1)(\rhonr)$ of $\Q_p(1)$ (see \cite{IV} and \cite{BC2}). One of the main technical difficulties in the computation of $R_{N/K}$ arises from the so-called cohomological term $C_{N/K}$, which requires the construction of a bounded complex $C_{N,T}^\bullet$ of cohomologically trivial modules which represents $R\Gamma(N,T)$ for a full $G_K$-stable $\Z_p$-sublattice $T$ of $V$. In this paper we generalize the construction of $C_{N,T}^\bullet$ in \cite[Thm.~2]{BC2} to the case of a higher dimensional $T$.
\end{abstract}

\begin{section}*{Introduction}
Let $K$ be a $p$-adic number field and let $\rhonr:G_K\to \Gl_r(\Z_p)$ be an unramified representation of the Galois group $G_K=\Gal(K^c/K)$. To such a representation we can associate the image of the Frobenius element $F_K$, i.e. the matrix $\rhonr(F_K)\in \Gl_r(\Z_p)$. We will see that this is actually a one to one correspondence between unramified $r$-dimensional representations of $G_K$ and matrices in $\Gl_r(\Z_p)$. We will be mainly interested in the case where $\rhonr$ is the restriction to $G_K$ of some unramified representation $\rhoQpnr:G_{\Q_p}\to \Gl_r(\Z_p)$, which, as above, can be described by a matrix $u\in \Gl_r(\Z_p)$. More concretely, our object of study will be the $G_K$-module $T=\Z_p^r(1)(\rhonr)$, i.e. the twist of $\Z_p^r$, considered with the trivial action, by the cyclotomic character $\chi_\mathrm{cyc}$ and by $\rhonr$.

The module $T$ arises in a very natural way as the Tate module of a particular formal group, as follows. By \cite[Sec.~13.3]{Hazewinkel78} there is a unique $r$-dimensional Lubin-Tate formal group $\calF = \calF_{pu^{-1}}$ attached to the parameter $pu^{-1}$. We will show that the $p$-adic Tate module $T_p\calF$ of $\calF$ is isomorphic to $T$ over the completion $\overline{\Qpnr}$ of the maximal unramified extension of $\Q_p$.

A module $T$ as above arises also in a more geometrical setting. If $A/\Q_p$ is an
abelian variety of dimension $r$ with good ordinary reduction, then we will show that the Tate module of the associated formal group $\hat{A}$ is isomorphic to $\Z_p^r(1)(\rhoQpnr)$ for an appropriate choice of $\rhoQpnr$. Here it is worth to remark that the converse is not true, i.e. not every module $\Z_p^r(1)(\rhoQpnr)$ comes from an abelian variety.

Let $N$ be a Galois extension of $K$ with Galois group $G$ and let $R\Gamma(N,T)$ be the complex of the $G_N$-invariants of the standard resolution of $T$. In this paper we construct a bounded complex of cohomologically trivial $G$-modules which represents $R\Gamma(N,T)$, generalizing \cite[Thm.~2]{BC2} to the higher dimensional case $r>1$. More concretely we will define $\Nnr$ similarly to \cite{BC2}, endow it with a natural action of $\Gal(\Kur/K) \times G$ and prove the following theorem.

\begin{thm_intro}\label{mainthm}
The complex
\[C^\bullet_{N, T} := \left[ \Nnr \xrightarrow{(F-1)\times 1} \Nnr \right]\]
with non-trivial modules only in degree $1$ and $2$ represents $R\Gamma(N,T)$.
\end{thm_intro}

Specializing our results to $r=1$, we will of course recover the corresponding results of \cite{BC2}, and as a by-product we will eliminate the distinction between the cases $\chinr|_{G_N}=1$ and $\chinr|_{G_N}\neq 1$, which had been studied separately in \cite{BC2}.

The main motivation for studying $R\Gamma(N,T)$ comes from the equivariant epsilon constant conjecture, which we denote by $\CEP$, where $V=\Q_p\otimes_{\Z_p}T$. Here the subscript
EP is short for Euler-Poincar\'e and refers to the fact that the conjecture
is formulated in terms of certain Euler-Poincar\'e characteristics. The
superscript na was added by Izychev and Venjakob for their non-abelian
generalization. For a formulation see for example Conjecture 3.1.1
in \cite{BC2}; for more details  and some remarks on the history of the conjecture we refer the interested reader to the introduction and Section 3.1 of \cite{BC2} or to the introduction of \cite{BC}; see also \cite{BurnsFlach98} and \cite{FukKato}, where these conjectures have been formulated first. The main idea in \cite{BC2} was to translate the conjecture to the language of $K$-theory in the spirit of \cite{BreuPhd} and prove it to be equivalent to the vanishing of an element
\[\begin{split}  R_{N/K}&= C_{N/K} + \Ucris+rm\partial^1_{\Z_p[G], \BdR[G]}(t) - m U_{tw}(\rho^{nr}_{\Q_p}) \\
&\qquad\qquad- r U_{N/K} + \partial^1_{\Z_p[G], \BdR[G]}(\epsilon_D(N/K,V))
\end{split}\]
in the relative algebraic $K$-group $K_0(\Z_p[G],\BdR[G])$.

We recall that for $r=1$ and representations $\rhonr$ which are restrictions of unramified extensions $\rhoQpnr : G_{\Q_p} \lra \Z_p^\times$ Izychev and Venjakob in \cite{IV} have proven the validity of $\CEP$ for tame extensions $N/K$. The main result of \cite{BC2} shows 
that $\CEP$ holds for certain weakly and wildly ramified finite abelian extensions $N/K$. We will generalize both results in an upcoming work with Werner Bley:

\begin{thm_intro}\label{intro thm 1}
Let $N/K$ be a tame extension of $p$-adic number fields and let
\[\rhoQpnr \colon G_{\Q_p} \lra \Gl_r(\Z_p)\]
be an unramified representation of $G_{\Q_p}$.
Let $\rhonr$ denote the restriction of $\rhoQpnr$ to $G_K$. Then $\CEP$ is true for $N/K$ and $V = \Q_p^r(1)(\rhonr)$,
if $\det(\rhonr(F_N) - 1) \ne 0$.
\end{thm_intro}

In particular, this includes the case in which $T$ is the $p$-adic Tate module of the formal group of an abelian variety defined over $\Q_p$. We will also prove $\CEP$ in some special cases of wild ramification, including for example the following statement.

\begin{thm_intro}\label{intro thm 2}
Let $p$ be an odd prime, let $A/\Q_p$ be an $r$-dimensional abelian variety with good ordinary reduction and let $K/\Q_p$ be the unramified extension of degree $m$. Let $N/K$ be a weakly and wildly ramified finite abelian extension with cyclic ramification group. Let $d$ denote the inertia degree of $N/K$ and assume that $m$ and $dp^{r-1}t(r)$ are relatively prime, where $t = t(r) := \prod_{i=1}^{r} \left( p^{i} - 1 \right)$ or $t=1$ according to whether $r>1$ or $r=1$. Then $\CEP$ is true for $N/K$ and $V = \Q_p\otimes_{\Z_p}T_p\hat A$, where $T_p\hat A$ is the Tate module of the formal group of $A$.
\end{thm_intro}

The results of the present paper will play a central role in the proof of both theorems. Actually, in the weakly ramified setting we will prove a slightly more general statement, but we renounce to formulate it here in order to avoid some technicalities.

\textbf{Notation.} We will mostly rely on the notation of \cite{BC2}. In particular $L/K$ will always denote a finite Galois extension, $L_0$ will be the completion of the maximal unramified extension $L^\mathrm{nr}$ of $L$ and $\widehat{L_0^\times}$ the $p$-completion of $L_0^\times$. We will denote by $e_{L/K}$ and $d_{L/K}$ the ramification index and the inertia degree of $L/K$, $\oo_L$ will be the ring of integers of $L$ and $U_L$ will be its group of units. We also set $\Lambda_L=\prod_r \widehat{L_0^\times}(\rhonr)$, $\Upsilon_L=\prod_r \widehat{U_{L_0}}(\rhonr)$ and $\zz=\Z_p^r(\rhonr)$ and we will use an additive notation for the (twisted) action of the absolute Galois group $G_L$. The elements fixed by the action of $G_L$ will be denoted by $\Lambda_L^{G_L}$, $\Upsilon_L^{G_L}$ and $\zz^{G_L}$ respectively.

If $A$ is a ring (and so in particular also if it is a field), then we denote $A^r=\{(a_1,\dots,a_r)\ |\ a_i\in A\}$; if $A$ is a multiplicative group, then we use the same notation for $A^r=\{a^r\ |\ a\in A\}$. Context will be strong enough to prevent any confusion. All vectors are intended to be columns, even if we write them horizontally for typographical reasons.

Let $\varphi$ be the absolute Frobenius automorphism, let $F_L$ be the Frobenius automorphism of $L$ and let $F=F_K$ be the Frobenius of $K$. We will denote $\calM_L=\rhonr(F_L)-I$.

\end{section}

\begin{section}[Unramified p-adic representations]{Unramified $p$-adic representations}\label{unramrep}

The aim of this section is to give a characterization of the unramified representations of $G_K$.

\begin{lemma}\label{GLZpOrders}
For $r, n \ge 1$ one has
\[\# \Gl_r(\Z_p / p^n\Z_p) = p^{(n-1)r^2} \cdot p^{s} \cdot  t\]
with
\[s = s(r) := (r-1)r / 2, \quad t = t(r) := \prod_{i=1}^{r} \left( p^{i} - 1 \right). \]
\end{lemma}

\begin{proof}
For $n=1$
\[\# \Gl_r(\Z_p / p\Z_p) = \prod_{j=0}^{r-1} \left( p^{r} - p^j \right) = p^{\sum_{j=0}^{r-1} j}\prod_{i=1}^{r} \left( p^{i} - 1 \right) = p^s \cdot t.\]
Each element of $\Gl_r(\Z_p / p\Z_p)$ has $p^{(n-1)r^2}$ different lifts to $\Gl_r(\Z_p / p^n\Z_p)$.
\end{proof}

\begin{lemma}\label{NV}
Each matrix $U \in \Gl_r(\Z_p)$ can be uniquely written in the form
$U = NV$ with matrices $N, V \in \Gl_r(\Z_p)$ such that
\[N^t = 1, \quad V^{p^s} \equiv 1 \pmod{p} \text{ and } NV = VN.\]
Moreover, we can further decompose $N$ into a product of matrices $N_q$ of $q$-power order, one for each $q$ dividing $t$, which commute pairwise and with $V$.
\end{lemma}

\begin{proof}
For all $n\geq 1$ there exist integers $a_n,b_n$ such that $a_np^{s+(n-1)r^2}+b_nt=1$. Note that $a_np^{s+(n-1)r^2}\equiv 1\equiv a_{n+1}p^{s+nr^2}\pmod{t}$ and hence also $a_n\equiv a_{n+1}p^{r^2}\pmod{t}$. Since $U^{p^{s+(n-1)r^2}}$ has order dividing $t$ modulo $p^n$ we deduce that
\[U^{a_np^{s+(n-1)r^2}}=(U^{p^{s+(n-1)r^2}})^{a_n}\equiv (U^{p^{s+(n-1)r^2}})^{a_{n+1}p^{r^2}}=U^{a_{n+1}p^{s+nr^2}}\pmod{p^n}.\]
Similarly $b_nt\equiv b_{n+1}t\pmod{p^{s+(n-1)r^2}}$, i.e. $b_n\equiv b_{n+1}\pmod{p^{s+(n-1)r^2}}$, and we obtain
\[U^{b_nt}\equiv U^{b_{n+1}t}\pmod{p^n}.\]
Therefore we can define
\[N:=\lim_{n\to\infty} U^{a_np^{s+(n-1)r^2}}\qquad\text{and}\qquad V:=\lim_{n\to\infty} U^{b_nt}.\]
It is now clear that $V^{p^s} \equiv 1 \pmod{p}$ and that all the other requirements hold modulo $p^n$ for any $n\geq 1$, hence they must also hold in $\Z_p$.

The last assertion can be proved similarly, but actually the proof is easier since the order of $N$ is finite. The pairwise commutativity follows from the fact that all the $N_q$ are some powers of $N$.

It remains to prove uniqueness. Let $N_1,V_1$ and $N_2,V_2$ be as above and let $a_n\in\N$ be such that $ta_n\equiv 1\pmod{p^{s+(n-1)r^2}}$. Then 
\[V_1\equiv V_1^{ta_n}=(N_1V_1)^{ta_n}=(N_2V_2)^{ta_n}=V_2^{ta_n}\equiv V_2\pmod{p^n}.\]
Since this is true for all $n$, we deduce that $V_1=V_2$ and hence also that $N_1=N_2$.
\end{proof}

\begin{lemma}\label{phiqp}
For $q\neq p$ there is a one to one correspondence of $\mathrm{Hom}_{\mathrm{cont}}(\Z_q,\Gl_r(\Z_p))$ with the set of all elements in $\Gl_r(\Z_p)$ of $q$-power order. This correspondence associates $\phi\in \mathrm{Hom}_{\mathrm{cont}}(\Z_q,\Gl_r(\Z_p))$ with $\phi(1)\in \Gl_r(\Z_p)$. In particular if $q\nmid p\cdot t(r)$, then $\mathrm{Hom}_{\mathrm{cont}}(\Z_q,\Gl_r(\Z_p))=0$.
\end{lemma}

\begin{proof}
Let $\phi\in \mathrm{Hom}_{\mathrm{cont}}(\Z_q,\Gl_r(\Z_p))$. By continuity, for any $n\geq 1$, the preimage of $\{M\in \Gl_r(\Z_p)\mid M\equiv I\pmod{p^n}\}$ must be a neighborhood of $0$ in $\Z_q$. In other words there must exist $m\geq 1$ such that $\phi(q^m)=\phi(1)^{q^m}\equiv I\pmod{p^n}$ and we can clearly assume that $q^m$ is always the maximal power of $q$ dividing $t$. Hence there exists a fixed $m$ such that for every $n\geq 1$ we have $\phi(1)^{q^m}\equiv I\pmod{p^n}$ and it follows that $\phi(1)^{q^m}=I$.

Conversely, let $U\in \Gl_r(\Z_p)$ be of order $q^m$, then we can set $\phi(x)=U^x$ for all $x\in\Z_q$. The two maps are clearly inverse to one another.
\end{proof}

\begin{lemma}\label{phipp}
There is a one to one correspondence of $\mathrm{Hom}_{\mathrm{cont}}(\Z_p,\Gl_r(\Z_p))$ with the set of all elements in $\Gl_r(\Z_p)$ which project to elements of $p$-power order in $\Gl_r(\Z_p/p\Z_p)$. This correspondence associates $\phi\in \mathrm{Hom}_{\mathrm{cont}}(\Z_p,\Gl_r(\Z_p))$ with $\phi(1)\in \Gl_r(\Z_p)$.
\end{lemma}

\begin{proof}
Let $\phi\in \mathrm{Hom}_{\mathrm{cont}}(\Z_p,\Gl_r(\Z_p))$. As in the proof of the previous lemma, we deduce that $\phi(1)$ is of $p$-power order modulo $p$.

Conversely let $U\in \Gl_r(\Z_p)$ be such that its order modulo $p$ is some power of $p$. By Lemma \ref{GLZpOrders} the order of $U$ modulo $p$ must divide $p^s$ and, moreover, its order modulo $p^n$ must divide $p^{s+(n-1)r^2}$. Therefore for $x\in\Z_p$, $U^x$ is well-defined modulo $p^n$ for all $n$ and hence it is also well-defined as an element of $\Gl_r(\Z_p)$. This implies that we can set $\phi(x)=U^x$.
\end{proof}

Let $\hat{\Z}$ denote the Pr\"ufer ring. We write
\[
\hat{\Z} = \prod_{q \nmid pt} \Z_q \times  \prod_{q \mid t} \Z_q \times \Z_p
\]
and write $1 \in \hat{\Z}$ in the form $1 = e' + e_t + e_p=e' + \sum_{q|t}e_q + e_p$ according to the above decomposition. In other words, the elements $e_p$ and $e_q$ in $\prod_{q \nmid pt} \Z_q \times  \prod_{q \mid t} \Z_q \times \Z_p$ correspond to $1\in\Z_p$ and $1\in\Z_q$ respectively.

\begin{lemma}
There is a one to one correspondence of $\mathrm{Hom}_{\mathrm{cont}}(\hat\Z,\Gl_r(\Z_p))$ with the set of all elements in $\Gl_r(\Z_p)$, which associates $\phi\in \mathrm{Hom}_{\mathrm{cont}}(\hat\Z,\Gl_r(\Z_p))$ with $\phi(1)\in \Gl_r(\Z_p)$.
\end{lemma}

\begin{proof}
By Lemma \ref{NV}, each matrix $U \in \Gl_r(\Z_p)$ can be decomposed as a product of the form $U = \left(\prod_{q|t}N_q\right)\cdot V$ with matrices $N_q, V \in \Gl_r(\Z_p)$, where $N_q$ has a $q$-power-order and the order of $V$ modulo $p$ is a power of $p$. Moreover all these matrices commute with one another. By Lemma \ref{phiqp} and Lemma \ref{phipp} we can define $\phi_q\in \mathrm{Hom}_{\mathrm{cont}}(\Z_q,\Gl_r(\Z_p))$ by $\phi_q(e_q)=N_q$ and $\phi_p\in \mathrm{Hom}_{\mathrm{cont}}(\Z_p,\Gl_r(\Z_p))$ by $\phi_p(e_p)=V$. Note that 
\[\mathrm{Hom}_{\mathrm{cont}}(\hat\Z,\Gl_r(\Z_p))=\prod_{q|t}\mathrm{Hom}_{\mathrm{cont}}(\Z_q,\Gl_r(\Z_p))\times \mathrm{Hom}_{\mathrm{cont}}(\Z_p,\Gl_r(\Z_p)).\]
So we can set $\phi=\prod_{q|t}\phi_q\cdot\phi_p$ and note that $\phi(1)=\prod_{q|t}\phi_q(e_q)\cdot\phi_p(e_p)=U$, which proves that the considered correspondence is one to one.
\end{proof}

We can reformulate this result as follows.
\begin{prop}\label{unram reps prop}
Let $K/\Q_p$ be a $p$-adic number field and let $U \in \Gl_r(\Z_p)$. 

a) Then the assignment $\rho(F_K) := U$ defines an
unramified representation 
\[\rho = \rho_{U,K} \colon G_K \lra \Gl_r(\Z_p).\] 

b) The map $U \mapsto \rho_{U,K}$ is a one-to-one correspondence between $\Gl_r(\Z_p)$ and $r$-dimensional
unramified $p$-adic representations of $G_K$. The  inverse is given by $\rho \mapsto \rho(F_K)$.
\end{prop}

\begin{coroll}
  Let $K$ be a $p$-adic number field and let $U \in \Gl_r(\Z_p)$. Then there exists a finite unramified extension
$E/K$ of degree dividing $t$ such that 
\[
\rho_U(F_E)^{p^{s(r)}} \equiv 1 \pmod{p}.
\]
\end{coroll}
\begin{proof}
Write $U = NV$ as above and let $m$ denote the order of $N$. Then $m$ divides $t(r)$ and the unramified extension $E$
of degree $m$ satisfies the assertions of the corollary since $F_E = F_K^m$ and so
$\rho_U(F_E)^{p^{s(r)}} = (NV)^{mp^{s(r)}} = V^{mp^{s(r)}} \equiv 1 \pmod{p}$.
\end{proof}

Next we will show that unramified representations appear naturally in the context of (higher dimensional) Lubin-Tate formal groups.

Let $u \in \Gl_r(\Z_p)$ and let $\rhoQpnr = \rho_u \colon G_{\Q_p} \lra \Gl_r(\Z_p)$ denote the unramified representation attached to $u$ by Proposition \ref{unram reps prop}. By \cite[Sec.~13.3]{Hazewinkel78} there is a unique $r$-dimensional Lubin-Tate formal group $\calF = \calF_{pu^{-1}}$ attached to the parameter $pu^{-1}$. We want to construct a formal isomorphism with the $r$-dimensional multiplicative group $\mathbb G_m^r$ over the completion $\overline{\Zpnr}$ of the ring of integers in the maximal unramified extension of $\Q_p$.

We first need the following lemma.
\begin{lemma}\label{trivcoh}
Let $K_m/\Q_p$ be the unramified extension of degree $m$ of $\Q_p$ and let $n\in\mathbb N$. Then
\[H^1(K_m/\Q_p,\Gl_r(\mathcal O_{K_m}/p^n\mathcal O_{K_m}))=1.\]
\end{lemma}

\begin{proof}
We prove this by induction on $n$. For $n=1$ the result is proved in  \cite[Ch.~X,~Prop.~3]{SerreLocalFields}. For the inductive step we consider the exact sequence
\[0\to M_r(p^n\mathcal O_{K_m}/p^{n+1}\mathcal O_{K_m})\to \Gl_r(\mathcal O_{K_m}/p^{n+1}\mathcal O_{K_m})\to \Gl_r(\mathcal O_{K_m}/p^n\mathcal O_{K_m})\to 1. \]
Note also that $p^n\mathcal O_{K_m}/p^{n+1}\mathcal O_{K_m}$ is isomorphic to the additive group of the finite field $\mathbb F_m=\mathcal O_{K_m}/p\mathcal O_{K_m}$. By the existence of a normal basis, the $\mathrm{Gal}(K_m/\Q_p)$-module $M_r(\mathbb F_m)$ is induced in the sense of \cite[Ch.~VII,~Sect.~1]{SerreLocalFields} and hence has trivial cohomology. Therefore, using the long exact cohomology sequence, from the inductive assumption $H^1(K_m/\Q_p, \Gl_r(\mathcal O_{K_m}/p^{n}\mathcal O_{K_m}))=1$ we deduce that 
\[H^1(K_m/\Q_p, \Gl_r(\mathcal O_{K_m}/p^{n+1}\mathcal O_{K_m}))=1.\]
\end{proof}

\begin{lemma}\label{epsilonE}
There exists $\epsilon\in \Gl_r(\overline{\Zpnr})$ such that $\varphi(\epsilon^{-1})\epsilon=u^{-1}$. Moreover if $\epsilon$ and $E$ are two solutions modulo $p^n$, then
\[\varphi(E^{-1} \epsilon)\equiv E^{-1} \epsilon\pmod{p^n}.\]
\end{lemma}

\begin{proof}
For all $n\in\mathbb N$, let $m(n)\in\mathbb N$ be the order of $\overline u\in \Gl_r(\mathcal \Z/p^n\Z)$ and let $K_{m(n)}/\Q_p$ be the unramified extension of degree $m(n)$ of $\Q_p$. Then
\[\prod_{i=0}^{m(n)-1}\varphi^i(u^{-1})=u^{-m(n)}\equiv I\pmod{p^n}.\]
It follows that the map $a:\mathrm{Gal}(K_{m(n)}/\Q_p)\to \Gl_r(\mathcal \Z/p^n\Z)$ defined by $a(\varphi^j)=\prod_{i=0}^{j-1}\varphi^i(\bar u^{-1})=\bar u^{-j}$ is a cocycle. By Lemma \ref{trivcoh} it must therefore also be a coboundary, i.e. there exists a matrix $\bar\epsilon\in \Gl_r(\mathcal O_{K_{m(n)}}/p^n\mathcal O_{K_{m(n)}})$ satisfying $\overline u^{-1}=a(\varphi)=\varphi(\bar\epsilon^{-1})(\bar\epsilon^{-1})^{-1}$. Therefore for every $n>0$ the set
\[V_n=\{\bar\epsilon\in \Gl_r(\mathcal O_{K_{m(n)}}/p^n\mathcal O_{K_{m(n)}}):\ \varphi(\bar\epsilon^{-1})\cdot\bar\epsilon=\overline u^{-1}\}\]
is non-empty. We want to define a map $\pi_{n+1,n}:V_{n+1}\to V_n$. Let us consider the projection $\tilde\pi_{n+1,n}:\Gl_r(\mathcal O_{K_{m(n+1)}}/p^{n+1}\mathcal O_{K_{m(n+1)}})\to \Gl_r(\mathcal O_{K_{m(n+1)}}/p^n\mathcal O_{K_{m(n+1)}})$ and let $\bar\epsilon\in V_{n+1}$. We want to show that $\tilde\pi_{n+1,n}(\bar\epsilon)\in V_n$, in order to define $\pi_{n+1,n}(\bar\epsilon)=\tilde\pi_{n+1,n}(\bar\epsilon)$. Since $V_n\neq\emptyset$, we can choose $\bar E\in V_n$. Then $\varphi(\bar\epsilon^{-1})\cdot\bar\epsilon\equiv u^{-1}\equiv  \varphi(\bar E^{-1})\cdot\bar E\pmod{p^n}$ and we deduce that
\[\varphi(\bar E\bar\epsilon^{-1})\equiv \bar E \bar \epsilon^{-1}\pmod{p^n}.\] 
It follows that $\bar E\bar \epsilon^{-1}\equiv c\pmod{p^n}$ for some $c\in \Gl_r(\Z_p)$. Therefore
\[\bar \epsilon^{-1}=\bar E^{-1} \bar E\bar \epsilon^{-1}\equiv\bar E^{-1} c\pmod{p^n},\]
i.e. $c^{-1}\bar E\in \mathcal O_{K_{m(n)}}$ represents the same element in $\Gl_r(\mathcal O_{K_{m(n+1)}}/p^n\mathcal O_{K_{m(n+1)}})$ as $\bar \epsilon$, which proves that $\tilde\pi_{n+1,n}(\bar\epsilon)\in V_n$.

We have shown that the $V_n$ form a projective system of non-empty finite sets and by standard arguments there must exist an element in their projective limit. This element $\epsilon$ satisfies the conditions in the statement of the lemma.
\end{proof}

\begin{prop}
There is an isomorphism $\theta:\calF\to\mathbb G_m^r$ defined over $\overline\Zpnr$ such that
\[\theta(X) = \epsilon^{-1} X + \ldots \quad \text{and}\quad \theta^\varphi \circ \theta^{-1} = u^{-1}.\]
\end{prop}

\begin{proof}
Using the matrix $\epsilon\in \Gl_r(\overline{\Zpnr})$ defined in Lemma \ref{epsilonE}, this is a higher dimensional version of \cite[Korollar V.2.3]{Neukirch92}.
\end{proof}

\begin{prop}\label{Lubin Tate lemma}
Let $T := \Z_p^r(1)(\rhonr)$ and, as usual, we write $T_p\calF$ for the $p$-adic Tate module of $\calF$. Then the homomorphism $\theta$ induces an isomorphism of  $p$-adic representations,
\[
T \cong T_p\calF.
\]
\end{prop}

\begin{proof}
Clearly $T \cong T_p\Gm^r(\rhonr)$. Let
\begin{eqnarray*}
  \psi \colon T_p\calF &\lra&  T_p\Gm^r(\rhonr), \\
z = \{ z_n \}_{n=1}^\infty &\mapsto& \psi(z) := \{ \theta(z_n) \}_{n=1}^\infty.
\end{eqnarray*}
Then, for $\tau \in I_{\Q_p}$, the fact that $\theta$ is defined over $\overline{\Qpnr}$  implies that
\[
\tau \cdot \psi(z) = \{ \theta(\tau(z_n)) \}_{n=1}^\infty = \psi(\tau(z)),
\]
and for $\varphi = F_{\Q_p}$ we have
\[
\varphi \cdot \psi(z) = \{ u\theta^\varphi(\varphi(z_n)) \}_{n=1}^\infty = 
 \{ (\theta\circ\theta^{-\varphi}\circ\theta^\varphi)(\varphi(z_n)) \}_{n=1}^\infty = 
\{ \theta(\varphi(z_n)) \}_{n=1}^\infty = \psi(\varphi(z)).
\]
\end{proof}

We conclude this section by choosing a more geometric point of view in order to show that a module $T=\Z_p^r(1)(\rhonr)$ can also arise in a natural way from an abelian variety. More specifically, let $A/\Q_p$ be an abelian variety of dimension $r$ with good ordinary reduction, then the associated formal group $\calF = \hat{A}$ is toroidal (see \cite{LubinRosen}) in the sense that there exists a unique isomorphism of formal groups $f \colon \calF \lra \Gm^r$ defined over the maximal unramified extension $\Qpnr/\Q_p$.
We let $u \in \Gl_r(\Z_p)$ denote the twist matrix defined by 
\begin{equation*}
\xymatrix{
\Gm^r\ar[r]^{f^{-1}} \ar@/_2pc/[rrr]_{u} & \hat{A} \ar[rr]^{f^{F_{\Q_p}}} && \mathbb{G}_m^r, 
}
\end{equation*}

\begin{prop}
Let $\rhoQpnr \colon G_{\Q_p} \lra \Gl_r(\Z_p)$ denote the unramified $p$-adic representation defined by $\rhoQpnr(F_{\Q_p}) := u^{-1}$. 
Then we have:
\begin{enumerate}
\item[(a)] The $p$-adic representations $T_p\hat{A}$ and $\Z_p^r(1)(\rhoQpnr)$ are isomorphic (as representations of $G_{\Q_p}$).
\item[(b)] Set $U := u^{d_K}$ and $\rhonr := \rhoQpnr|_{G_K}$. Then $\rhonr = \rho_{U^{-1}, K}$ and the $p$-adic representations $T_p\hat{A}$ and 
$\Z_p^r(1)(\rhonr)$ are isomorphic (as representations of $G_K$).
\end{enumerate}
\end{prop}

\begin{proof}
  As in the proof of Proposition \ref{Lubin Tate lemma}.
\end{proof}

\end{section}

\begin{section}{Some preliminary results}\label{preliminary results}

In this section we formulate and prove some technical lemmas, which will be needed in the proof of Theorem \ref{mainthm}. We start with a generalized version of \cite[Lemma 4.1.1]{BC2}.

\begin{lemma}\label{UL1hat}
Let $L/K$ denote a finite Galois extension. Then there is an isomorphism
\[\fFL\colon\calF(\p_L)\times \zz^{G_L} \lra \Lambda_L^{G_L}.\]
Explicitly $\fFL$ is given by a power series of the form $1+\epsilon^{-1} X+(\deg\geq 2)$ on $\calF(\p_L)$, where $\epsilon\in \Gl_r(\overline{\Zpnr})$ is as above (here $1$ is a vector consisting only of ones).
\end{lemma}

\begin{proof}
We have $L_0^\times \cong \pi_L^\Z \times \kappa^\times \times U_{L_0}^{(1)}$, where $\kappa$ denotes the residue class field of $L_0$. Since any element of $\kappa^\times$ has coprime to $p$ order and $U_{L_0}^{(1)}$ is $p$-complete we obtain $\widehat{L_0^\times} \cong \Z_p \times U_{L_0}^{(1)}$.
By \cite[Lemma~page~237]{LubinRosen} and recalling the obvious isomorphism $U_{L_0}^{(1)}\cong \mathbb G_m(\p_{L_0})$, the formal isomorphism $\theta:\calF\to\mathbb G_m^r$ induces an isomorphism
\[\tilde f\colon\calF(\p_L)\lra \left(\prod_r U_{L_0}^{(1)}(\rhonr) \right)^{G_L}.\]
By the construction of $\theta$, this isomorphism is given by a power series $1+\epsilon^{-1} X+(\deg\geq 2)$. To conclude, it is enough to twist and take $G_L$-fixed points in $\prod_r\widehat{L_0^\times} \cong \prod_r\Z_p \times \prod_r U_{L_0}^{(1)}$.
\end{proof}

We set $\calG := \Gal(K_0/K) \times \Gal(L/K)$. In analogy to \cite[(6.2)]{Chinburg85} and \cite{BC2} we set
\[\Lnr := \Ind_{\Gal(L_0/K)}^\calG  \Lambda_L\]
and we usually identify $\Lnr$ with $\prod_{d_{L/K}}\prod_r \widehat{L_0^\times}$. Note that the $\calG$-module structure on $\Lnr$ is characterized by
\begin{equation}\label{rule 1}
(F\times 1)\cdot[x_1,x_2,\dots,x_{d_{L/K}}]=[F_Lx_{d_{L/K}},x_1,x_2,\dots,x_{d_{L/K}-1}],
\end{equation}
and
\begin{equation}\label{rule 2}
(F^{-n}\times \sigma)\cdot[x_1,x_2,\dots,x_{d_{L/K}}]=[\tilde\sigma x_1,\tilde\sigma 
x_2,\dots,\tilde\sigma x_{d_{L/K}}],
\end{equation}
where $x_i\in\Lambda_L$, the elements $F^{-n}$ and $\sigma\in \Gal(L/K)$ have the same restriction to $L \cap K_0$ and $\tilde\sigma \in \Gal(L_0/K)$ is uniquely defined by $\tilde\sigma|_{K_0}=F^{-n}$ and $\tilde\sigma|_L=\sigma$. 

We also have a $G_K$- and a $\Gal(L/K)$-action on $\Lnr$, via the natural maps $G_K\to\Gal(L_0/K)\hookrightarrow\calG$ and $\Gal(L/K) \hookrightarrow \calG$. The following lemma clarifies the structure of $\Lnr$ as a $\Gal(L/K)$-module. 

\begin{lemma}\label{LurInd}
If we identify the inertia group $I_{L/K}$ with $\Gal(L_0/K_0)$, then there is a $\Gal(L/K)$-isomorphism
\[\tfFL \colon \Ind_{I_{L/K}}^{\Gal(L/K)}\left(\prod_r\widehat{L_0^\times}\right) \lra \Lnr.\]
In particular, $\Lnr$ is $\Gal(L/K)$-cohomologically trivial.
\end{lemma}

\begin{proof}
The proof is the same as in \cite[Lemma 4.1.2]{BC2}, up to some obvious modifications.
\end{proof}

The following lemma will be used to prove Lemma \ref{Neukirch}, which will be one of the key ingredients in the proof of our main theorem.
\begin{lemma}\label{linalg}
Let $K$ be an algebraically closed field of characteristic $p$, let $q$ be a power of $p$, let $A$, $B\in K^{n,n}$ for some $n\in\N$ and let $f:K\to K^n$ be a function whose components are polynomials involving only terms of degree larger than $q$ and constants. Assuming that $B$ is invertible, the system of equations
\begin{equation}\label{eqlinalg}f(x_1)+Ax^q+Bx=0\end{equation}
has a solution $x=(x_1,\dots,x_n)\in K^n$, where the $q$-th power $x^q$ is taken componentwise.
\end{lemma}

\begin{proof}
We prove this lemma by induction on $n$. For $n=1$, (\ref{eqlinalg}) consists of a single non-constant polynomial equation, since there must be a non-trivial term of degree $1$ ($B$ is assumed to be invertible). Since $K$ is algebraically closed, it follows that (\ref{eqlinalg}) has a solution.

So let us assume the result holds for $n-1\geq 1$ and let us prove it for $n$. Up to some elementary row operations we can assume that the first row of $A$ has at most one nonzero entry $a_{1,1}$ (note that we allow also $a_{1,1}=0$). Then the first equation of our system takes the form
\[f_1(x_1)+a_{1,1}x_1^q+\sum_{j=1}^nb_{1,j}x_j=0,\]
where some of the $b_{1,j}$ are non-zero. Here we need to distinguish two cases.

Case 1: $b_{1,j}=0$ for all $j>1$ and $b_{1,1}\neq 0$. This means that $B$ takes the form
\[B=\begin{pmatrix}b_{1,1}& 0\\ \tilde b &\tilde B\end{pmatrix},\]
where $\tilde B\in K^{n-1,n-1}$ is invertible and $\tilde b\in K^{n-1,1}$. In particular, the first equation is a non-trivial polynomial in the only variable $x_1$, which has a solution since $K$ is algebraically closed. Substituting this value for $x_1$ in all the other equations we obtain a new system of the shape
\[c+\tilde A\tilde x^q+\tilde B\tilde x=0,\]
where $\tilde x=(x_2,\dots,x_n)$ and $c$ is a constant vector in $K^{n-1}$, which we can interpret as a constant function $K\to K^{n-1}$. Note also that $\tilde B$ is invertible as observed above. Therefore we can apply the inductive hypothesis to find a solution $\tilde x$; together with the already computed first component $x_1$, we obtain a solution of (\ref{eqlinalg}).

Case 2: $b_{1,j}\neq 0$ for some $j>1$. Up to a permutation of the variables and normalizing the last coefficient, we can assume $b_{1,n}=1$. Up to some further elementary row operations (which do not modify the first row) we may assume that $b_{i,n}=0$ for $i>1$. So the matrix $B$ has the shape
\[B=\begin{pmatrix}\tilde b& 1\\ \tilde B & 0\end{pmatrix},\]
where $\tilde B\in K^{n-1,n-1}$ is clearly invertible and $\tilde b\in K^{1,n-1}$. Hence we can rewrite the first equation as
\begin{equation}\label{defxn}x_n=-f_1(x_1)-a_{1,1}x_1^q-\sum_{j=1}^{n-1} b_{1,j}x_j.\end{equation}
Substituting this into the other equations of (\ref{eqlinalg}) we obtain a new system of the type
\begin{equation}\label{eqlinalgindu}\tilde f(x_1)+\tilde A\tilde x^q+\tilde B\tilde x=0,\end{equation}
where $\tilde x=(x_1,\dots,x_{n-1})$ and $\tilde B$ is as above. Note that (\ref{eqlinalgindu}) has the same shape as (\ref{eqlinalg}), since, as already mentioned, $\tilde B$ is invertible. Note also that here we used that $q$ is a multiple of the characteristic of the field, so that, when computing $x_n^q=(-f_1(x_1)-a_{1,1}x_1^q-\sum_{j=1}^{n-1} b_{1,j}x_j)^q$, all the mixed terms vanish. Therefore we can use the inductive hypothesis to find $x_1,\dots,x_{n-1}$ and then take $x_n$ as in (\ref{defxn}). This will be a solution of (\ref{eqlinalg}).
\end{proof}

Now we are in a position to prove the following result which generalizes a well-known lemma, see e.g. \cite[Lemma V.2.1]{Neukirch92}.

\begin{lemma}\label{Neukirch}
Given an element $c\in \Upsilon_L$, there exists $x\in \Upsilon_L$ such that $(F_L-1)\cdot x=c$. If $\tilde x$ is such that $(F_L-1)\cdot \tilde x\equiv c\pmod{\p_{L_0}^i}$ for some $i$, then we can assume $x\equiv\tilde x\pmod{\p_{L_0}^i}$. In particular, if $c\in \prod_r U_{L_0}^{(1)}(\rhonr)$, then we can assume $x\in \prod_r U_{L_0}^{(1)}(\rhonr)$.
\end{lemma}

\begin{proof}
Let $I$ be the $r\times r$ identity matrix and let $q\in\N$ be such that 
$x^{F_L}\equiv x^q\pmod{\p_{L_0}}$ for all $x\in\oo_{L_0}$. The determinant of the matrix $q\rhonr(F_L)-I$ is obviously a unit in $\Z_p$. Therefore there exists an inverse matrix $\calM$ of $q\rhonr(F_L)-I$ with coefficients in $\Z_p$. Then we take $x_1=\calM\cdot c$ and we have $(F_L-1)\cdot x_1\equiv c\pmod{\p_{L_0}}$.

Let us assume we have $x_i$ such that $(F_L-1)\cdot x_i\equiv c\pmod{\p_{L_0}^i}$, i.e. such that there exists $a_i\in \prod_r U_{L_0}^{(i)}(\rhonr)$ with $((F_L-1)\cdot x_i)\cdot a_i=c$, where the product of $((F_L-1)\cdot x_i)$ and $a_i$ is done componentwise. We take $b_i\in \prod_r\mathcal O_{L_0}$ so that $a_i=1+\pi_L^ib_i$, where $\pi_L$ is an uniformizing element of $L$ (and hence also of $L_0$) and $1$ is the vector consisting only of ones. Since the residue field of $L_0$ is algebraically closed, we can apply Lemma \ref{linalg} and obtain a $y_i\in \prod_r\oo_{L_0}$ such that
\[\rhonr(F_L)y_i^q-y_i-b_i\equiv 0\pmod{\p_{L_0}}.\]
We set $x_{i+1}=x_i(1+y_i\pi_L^i)$ and we calculate
\[\begin{split}(F_L-1)\cdot x_{i+1}&=((F_L-1)\cdot x_i)\cdot ((F_L-1)\cdot(1+y_i\pi_L^i))\\
&\equiv ca_i^{-1}(1+\rhonr(F_L)y_i^q\pi_L^i)(1-y_i\pi_L^i)\pmod{\p_{L_0}^{i+1}}\\
&\equiv ca_i^{-1}(1+\rhonr(F_L)y_i^q\pi_L^i-y_i\pi_L^i)\pmod{\p_{L_0}^{i+1}}\\
&\equiv ca_i^{-1}(1+b_i\pi_L^i)\pmod{\p_{L_0}^{i+1}}\\
&\equiv c\pmod{\p_{L_0}^{i+1}}.\end{split}\]
Therefore we have an inductive construction for the $x_i$, which can be started with $x_1=1$ in the case that $c\in \prod_r U_{L_0}^{(1)}(\rhonr)$. Then the conclusions of the lemma hold for $x=\lim_{n\to\infty}x_n$.
\end{proof}

To conclude this section we prove some technical results which will be needed to prove our main theorem.

\begin{lemma}\label{linalglemma}
Let $R$ be a discrete valuation ring with uniformizing element $\pi$, let $K$ be its field of quotients and let $f\in\mathrm{End}_R(R^r)$. Then there exists $\omega\in\N$ such that for all $n\geq \omega$,
\[K\cdot\mathrm{im}(f)\cap \pi^nR^r\subseteq \pi^{n-\omega}\mathrm{im}(f).\]
\end{lemma}

\begin{proof}
We set $\mathcal U=\mathrm{im}(f)$ and $\tilde{\mathcal U}=R^r\cap K\cdot\mathcal U$. Since $R^r$ is a noetherian ring, the ascending chain of ideals $R^r\cap \frac{1}{\pi^n}\mathcal U$ for $n>0$ stabilizes. Therefore there exists $s\geq 0$ such that $\tilde{\mathcal U}=R^r\cap \frac{1}{\pi^s}\mathcal U$ and hence we have 
\[\pi^s\tilde{\mathcal U}\subseteq \mathcal U.\]
By the Theorem of Artin-Rees there exists $k\in\N$ such that for all $n\geq k$, 
\[\tilde{\mathcal U}\cap \pi^n R^r\subseteq \pi^{n-k}(\tilde{\mathcal U}\cap (\pi^k R^r)).\]
Therefore setting $\omega=s+k$ we deduce that for all $n\geq \omega$,
\[K\cdot\mathcal U\cap \pi^nR^r\subseteq \tilde{\mathcal U}\cap \pi^nR^r\subseteq \pi^{n-k}(\tilde{\mathcal U}\cap (\pi^k R^r))\subseteq \pi^{n-\omega}\pi^s\tilde{\mathcal U}\subseteq \pi^{n-\omega}\mathcal U.\]
\end{proof}

\begin{lemma}\label{pipartforW}
Keeping all the notations of the previous lemma, for all $m,n\in\N$ such that $m>n+\omega$ we have
\[((\mathrm{im}(f)\cap \pi^n R^r)+\pi^m R^r)/(\pi^n \mathrm{im}(f)+\pi^m R^r)\cong (\mathrm{im}(f)\cap \pi^n R^r)/\pi^n \mathrm{im}(f).\]
\end{lemma}

\begin{proof}
The natural map
\begin{equation}\label{surjforiso}\mathrm{im}(f)\cap \pi^n R^r\to ((\mathrm{im}(f)\cap \pi^n R^r)+\pi^m R^r)/(\pi^n \mathrm{im}(f)+\pi^m R^r)\end{equation}
is obviously surjective. Its kernel is
\[\mathrm{im}(f)\cap \pi^n R^r\cap (\pi^n \mathrm{im}(f)+\pi^m R^r).\]
We want to show that this is equal to $\pi^n \mathrm{im}(f)$. Note that
\[\pi^n \mathrm{im}(f)+\pi^m R^r\subseteq \pi^n R^r\]
and, recalling Lemma \ref{linalglemma}, we get
\[\begin{split}\mathrm{im}(f)\cap \pi^n R^r\cap (\pi^n \mathrm{im}(f)+\pi^m R^r)&=\mathrm{im}(f)\cap (\pi^n \mathrm{im}(f)+\pi^m R^r)\\
&=\pi^n \mathrm{im}(f)+(\mathrm{im}(f)\cap \pi^m R^r)\\
&=\pi^n\left(\mathrm{im}(f)+\left(\frac{1}{\pi^n}\mathrm{im}(f)\cap \pi^{m-n} R^r\right)\right)\\
&=\pi^n\mathrm{im}(f).\end{split}\]
\end{proof}

From now on we set $\omega_L$ to be the natural number obtained by setting $R=\Z_p$ and considering the endomorphism $f$ of $\Z_p^r$ induced by the matrix $\calM_L =\rhonr(F_L)-I$ in Lemma \ref{linalglemma}.

\begin{lemma}\label{MittagLefflerW}
If $m,n\in\N$ are such that $m>n+\omega_L$, then
\[\begin{split}&\left(\left((F_L-1)\cdot\Lambda_L\cap p^n\Lambda_L\right)+p^m\Lambda_L\right)/\left(p^n(F_L-1)\cdot\Lambda_L+p^m\Lambda_L\right)\\
&\qquad\cong \left((F_L-1)\cdot\Lambda_L\cap p^n\Lambda_L\right)/\left(p^n(F_L-1)\cdot\Lambda_L\right).\end{split}\]
\end{lemma}

\begin{proof}
Let $e:\zz\to\Lambda_L$ be defined by $e(a_1,\dots,a_r)=(\pi_L^{a_1},\dots,\pi_L^{a_r})$. Then
\[\Lambda_L=\Upsilon_L\times e(\zz)\]
and all the computations may be split into those two components. Let us consider $\Upsilon_L$ first. By Lemma \ref{Neukirch}, $(F_L-1)\cdot\Upsilon_L=\Upsilon_L$ and it is straightforward to check that
\[\begin{split}&\left(\left((F_L-1)\cdot\Upsilon_L\cap p^n\Upsilon_L\right)+p^m\Upsilon_L\right)/\left(p^n(F_L-1)\cdot\Upsilon_L+p^m\Upsilon_L\right)\\
&\qquad\cong \left((F_L-1)\cdot\Upsilon_L\cap p^n\Upsilon_L\right)/\left(p^n(F_L-1)\cdot\Upsilon_L\right),\end{split}\]
since both sides are trivial. So it remains the $e(\zz)$-component and what has to be proved is equivalent to the statement of Lemma \ref{pipartforW} for $\calM_L =\rhonr(F_L)-I$.
\end{proof}

\end{section}

\begin{section}{The proof of Theorem \ref{mainthm}}

The goal of this section is to prove our main theorem, generalizing the approach of \cite[Sect.~4.2~and~4.3]{BC2} to the case in which $T$ is an unramified twist of $\Z_p^r(1)$, for any $r\geq 1$.

Let $W_n$ and $V_n$ be respectively the kernel and the image of the map
\[F_L-1:\Lambda_L/p^n\Lambda_L\to \Lambda_L/p^n\Lambda_L.\]

\begin{lemma}\label{Neu comm diag}
For $m>n$, we have a commutative diagram of exact sequences:
\[\xymatrix{0\ar[rr]&&W_{m}\ar[rr]\ar[d]^{\tau_{m,n}}&&\Lambda_L/p^m\Lambda_L
\ar[rr]^-{F_L-1}\ar[d]^{\pi_{m,n}}   &&V_{m}\ar[rr]\ar[d]^{\sigma_{m,n}}&&0\\
0\ar[rr]&&W_{n}\ar[rr]&&\Lambda_L/p^n\Lambda_L\ar[rr]^-{F_L-1}&&V_{n}
\ar[rr]&&0.}\]
Here $\pi_{m,n}$ denotes the canonical projection and $\tau_{m,n}$ and $\sigma_{m,n}$ are induced by $\pi_{m,n}$. Moreover, $\pi_{m,n}$ and $\sigma_{m,n}$ are surjective and the projective system $\left( W_n \right)_n$ satisfies the Mittag-Leffler condition.
\end{lemma}

\begin{proof}
The exactness of the rows and the commutativity of the diagram as well as the surjectivity of $\pi_{m,n}$ and $\sigma_{m,n}$ are evident. It remains to show that the projective system $(W_n)_n$ satisfies the Mittag-Leffler condition. To that end we prove that $\coker(\tau_{m,n})$ stabilizes for any $n\in\N$, as $m$ tends to infinity.

By the Snake Lemma we obtain 
\[p^n\Lambda_L/p^m\Lambda_L\xrightarrow{F_L-1}
\ker(\sigma_{m,n})\lra\coker(\tau_{m,n})\lra 0\]
i.e.
\[\coker(\tau_{m,n})\cong\coker\left(p^n\Lambda_L/p^m\Lambda_L\xrightarrow{F_L-1}
\ker(\sigma_{m,n})\right).\]
Note that
\[V_n=\left((F_L-1)\cdot\Lambda_L+p^n\Lambda_L\right)/p^n\Lambda_L\]
and
\[\begin{split}\ker \sigma_{m,n}&=\left(\left((F_L-1)\cdot\Lambda_L+p^m\Lambda_L\right)\cap p^n\Lambda_L\right)/p^m\Lambda_L\\
&=\left(\left((F_L-1)\cdot\Lambda_L\cap p^n\Lambda_L\right)+ p^m\Lambda_L\right)/p^m\Lambda_L.\end{split}\]
Therefore
\[\begin{split}&\coker(\tau_{m,n})\cong \left(\left((F_L-1)\cdot\Lambda_L\cap p^n\Lambda_L\right)+ p^m\Lambda_L\right)/\left(p^n(F_L-1)\cdot\Lambda_L+p^m\Lambda_L\right).\end{split}\]
It is an immediate consequence of Lemma \ref{MittagLefflerW} that $\coker(\tau_{m,n})$ is independent of $m$, provided $m>n+\omega_L$.
\end{proof}

\begin{lemma}
We have a commutative diagram of short exact sequences
\[\xymatrix{
0\ar[r] & V_{m}\ar[rr]\ar[d]^{\sigma_{m,n}} && \Lambda_L/p^m\Lambda_L \ar[rr]^-{\nu_{L,m}}\ar[d]^{\pi_{m,n}}   
&& \zz/((F_L-1)\cdot\zz+p^m\zz) \ar[r]\ar[d] & 0\\
0\ar[r]&V_{n}\ar[rr]&&\Lambda_L/p^n\Lambda_L\ar[rr]^-{\nu_{L,n}}&&\zz/((F_L-1)\cdot\zz+p^n\zz)
\ar[r]&0,}\]
where $\nu_{L,m},\nu_{L,n}$ are the maps induced by the componentwise normalized valuation of $L^\times$.
The vertical maps are all surjective and so, in particular, they satisfy the Mittag-Leffler condition.
\end{lemma}

\begin{proof}
The commutativity of the diagram, the injectivity of the left inclusions and the surjectivity of the maps on the right and of the vertical ones are obvious.

Let us consider the homomorphism $e:\zz\to\Lambda_L$ defined by $e(a_1,\dots,a_r)=(\pi_L^{a_1},\dots,\pi_L^{a_r})$. Then, recalling Lemma \ref{Neukirch} and using an additive notation as usual,
\[\begin{split}\ker\nu_{L,n}&=\Upsilon_L/p^n\Upsilon_L\times e((F_L-1)\cdot\zz+p^n\zz)/e(p^n\zz)\\
&=\left((F_L-1)\cdot\Upsilon_L+p^n\Upsilon_L\right) /p^n\Upsilon_L \times ((F_L-1)\cdot e(\zz)+p^n e(\zz))/p^ne(\Z_p^r)\\
&=\left((F_L-1)\cdot\Lambda_L+p^n\Lambda_L\right) /p^n\Lambda_L=V_n.
\end{split}\]
Therefore there is exactness also at $\Lambda_L/p^m\Lambda_L$.
\end{proof}

We define
\[W := \varprojlim_n W_n, \quad V := \varprojlim_n V_n,\]
and thus, using the above results, we obtain by \cite[Prop.~3.5.7]{Weibel} two short exact sequences 
\[ 0 \lra W \lra \Lambda_L \lra V \lra 0, \]
\[ 0 \lra V \lra \Lambda_L \lra \zz/ (F_L-1)\cdot\zz \lra 0.\]
Splicing together these two sequences we obtain
\begin{equation}\label{twisted ses 117}
0 \lra W \lra \Lambda_L \xrightarrow{F_L - 1} \Lambda_L\stackrel{\nu_L}\lra  \zz / (F_L-1)\cdot\zz  \lra 0.
\end{equation}

We clearly have $W \cong \Lambda_L^{G_L}$ which by Lemma \ref{UL1hat} is isomorphic to $\calF(\p_L)\times\zz^{G_L}$.

Note that we also have a natural $\Gal(L/K)$-action on $\Lambda_L^{G_L}$ and on $\zz/ (F_L-1)\cdot \zz$, defined by extending an element $\sigma\in \Gal(L/K)$ to an element $\tilde\sigma\in\Gal(L_0/K)$.

So we can formulate the following theorem, which generalizes a result by Serre.

\begin{thm}\label{FpnZpomega}
We have an exact sequence
\[0\to\Lambda_L^{G_L}\longrightarrow \Lnr \xrightarrow{(F-1)\times 1}
\Lnr \stackrel{w_L}\longrightarrow\zz/(F_L-1)\cdot\zz\to 0\]
of $\Z_p[\Gal(L/K)]$-modules. Here $w_L$ denotes the sum of the componentwise valuations $\nu_L$.
\end{thm}

\begin{proof}
The calculations are very similar to those in the proof of \cite[Thm.~4.2.3]{BC2}. The main ingredients are the definitions and equation (\ref{twisted ses 117}).
\end{proof}

We notice that the above theorem can also be considered as an improvement of \cite[Thm.~4.2.3]{BC2} in the sense that it does not require any assumptions on $\rhonr(F_L)$. Therefore the distinction made in \cite{BC2} between the cases $\chinr(F_L)=1$ and $\chinr(F_L)\neq 1$ turns out to be unnecessary. This allows us also to formulate the next lemma without considering separately the different possibilities.

We let $M/L/K$ be finite extensions such that both $M/K$ and $L/K$ are Galois. We have a canonical  inclusion
\begin{eqnarray*}
\iota = \iota_{M/L} \colon \Lnr &\lra& \Mnr, \\ 
x &\mapsto&[F_L^{d_{M/L}-1}x,F_L^{d_{M/L}-2}x,\dots,x],
\end{eqnarray*} 
where $x=[x_1,\dots,x_{d_{L/K}}] \in \prod_{d_{L/K}}\Lambda_L$.

\begin{lemma}\label{diagforlimit}
For extensions $M/L/K$ there is a commutatiove diagram
\[\xymatrix@C-=0.5cm{0\ar[r]&\Lambda_L^{G_L}\ar[r]\ar[d]^{\subseteq}&\Lnr\ar[rr]^-{(F-1)\times 1}\ar[d]^{{\iota_{M/L}}}&&
\Lnr\ar[r]^-{w_L}\ar[d]^{{\iota_{M/L}}}&\zz/(F_L-1)\cdot\zz\ar[r]
\ar[d]^{e_{M/L}\sum_{i=0}^{d_{M/L-1}}F_L^i}&0\\
0\ar[r]&\Lambda_M^{G_M}\ar[r]&\Mnr\ar[rr]^-{(F-1)\times 1}&&\Mnr\ar[r]^-{w_M}&\zz/(F_M-1)\zz\ar[r]&0}\]
with exact rows, where $e_{M/L}$ and $d_{M/L}$ are the ramification index and the inertia degree of $M/L$.
\end{lemma}

\begin{proof}
The proof is a long but straightforward computation.
\end{proof}

Following \cite{BC2}, the next step is to take a direct limit in the previous diagram. In order to get explicit results, we need a couple of lemmas.

\begin{lemma}\label{tildeK}
There exists a finite extension $\tilde K/K$ such that $\calM_{\tilde K}\Q_p^r=\bigcap_{L/K}\calM_L\Q_p^r$. We can assume $N\subseteq \tilde K$.
\end{lemma}

\begin{proof}
Let us define a sequence of fields $L_i$ by induction, starting with $L_0=N$. We assume that we have already defined $L_i$ for some $i\in\N$ and that there exists a finite extension $L/K$ such that $\calM_{L_i}\Q_p^r\cap\calM_L\Q_p^r\subsetneq \calM_{L_i}\Q_p^r$. Then we define $L_{i+1}=L_i L$ and we note that $\calM_{L_{i+1}}\Q_p^r\subseteq \calM_{L_i}\Q_p^r\cap\calM_L\Q_p^r\subsetneq \calM_{L_i}\Q_p^r$. In particular this process must terminate since the dimension of $\calM_{L_i}\Q_p^r$ is finite and decreasing at each step. We call $\tilde K$ the last of the $L_i$. By construction it must satisfy the desired properties.
\end{proof}

From now on we will fix a field $\tilde K$ as in the previous lemma.

\begin{lemma}\label{sameasdegree}
Let $M/L$ be a finite extension of local fields containing $\tilde K$ and let $(a_1,\dots,a_r)\in\zz$. Then there exists an extension $\tilde M$ of $M$ such that the classes of the elements
\[e_{M/L}\sum_{i=0}^{d_{M/L}-1}F_L^i\cdot (a_1,\dots, a_r)\]
and
\[[M:L](a_1,\dots, a_r)\]
have the same image under the map
\[e_{\tilde M/M}\sum_{i=0}^{d_{\tilde M/M}-1}F_M^i:\zz/(F_M-1)\zz\to \zz/(F_{\tilde M}-1)\zz\]
of Lemma \ref{diagforlimit}.
\end{lemma}

\begin{proof}
We calculate
\[\begin{split}
&e_{M/L}\sum_{i=0}^{d_{M/L}-1}F_L^i\cdot (a_1,\dots,a_r)-[M:L](a_1,\dots,a_r)\\
&\qquad\qquad=e_{M/L}\sum_{i=0}^{d_{M/L}-1}(F_L^i-1)\cdot (a_1,\dots,a_r)\\
&\qquad\qquad=e_{M/L}(F_L-1)\sum_{i=0}^{d_{M/L}-1}\sum_{j=0}^{i-1}F_L^j\cdot (a_1,\dots,a_r)\\
&\qquad\qquad=e_{M/L}\calM_L\sum_{i=0}^{d_{M/L}-1}\sum_{j=0}^{i-1}\rhonr(F_L)^j (a_1,\dots,a_r).
\end{split}\]
Since $L$ and $M$ both contain $\tilde K$, it follows that $\calM_L\Q_p^r=\calM_M\Q_p^r$. Hence there exists a vector $v\in\Q_p^r$ such that
\[e_{M/L}\sum_{i=0}^{d_{M/L}-1}F_L^i\cdot (a_1,\dots,a_r)-[M:L](a_1,\dots,a_r)=\calM_Mv=(F_M-1)\cdot v.\]
Clearly, for some $e$, we have $p^ev\in\Z_p^r$. Let $\tilde M$ be a totally ramified extension of $M$ of degree $p^e$. Then we obtain
\[\begin{split}
&e_{\tilde M/M}\sum_{i=0}^{d_{\tilde M/M}-1}F_M^i\cdot \left(e_{M/L}\sum_{i=0}^{d_{M/L}-1}F_L^i\cdot (a_1,\dots,a_r)-[M:L](a_1,\dots,a_r)\right)\\
&\qquad=p^e\left(\sum_{i=0}^{d_{\tilde M/M}-1}F_M^i\right)(F_M-1)\cdot(v)=(F_{\tilde M}-1)\cdot(p^ev)\in (F_{\tilde M}-1)\cdot\zz.
\end{split}\]
\end{proof}

\begin{lemma}\label{QpMQp}
There is an isomorphism
\[\Q_p^r/\calM_{\tilde K}\Q_p^r(\rhonr)\cong \varinjlim_L\zz/(F_L-1)\cdot\zz,\]
where the morphisms of the direct system are those of Lemma \ref{diagforlimit}.
\end{lemma}

\begin{proof}
We consider a tower of totally ramified extensions $\tilde K\subseteq \tilde K_1\subseteq \tilde K_2\subseteq\dots$, each one of degree $p$. Then we define $\Phi:\Q_p^r\to \varinjlim_L\Z_p^r/\calM_L\Z_p^r(\rhonr)$ by
\[\Phi\left(\frac{a_1}{p^e},\dots,\frac{a_r}{p^e}\right)=((a_1,\dots,a_r)\pmod{\calM_{\tilde K_e}\Z_p^r})_{\tilde K_e},\]
where $a_1,\dots,a_r\in\Z_p$ and $e\in\N$. This is clearly well-defined. To prove surjectivity, let $((a_1,\dots,a_r)\pmod{\calM_L\Z_p^r})_L$, where we may assume that $L\supseteq \tilde K$. Let $e\in\N$ and $\alpha\in\Z$ be coprime to $p$ and such that $[L:\tilde K]=p^e\alpha$. By construction the element $((\alpha^{-1} a_1,\dots,\alpha^{-1} a_r)\pmod{\calM_{\tilde K_e}\Z_p^r})_{\tilde K_e}$ is in the image of $\Phi$ and, by Lemma \ref{sameasdegree}, it represents the same element as
\[((\alpha^{-1} [L\tilde K_e:\tilde K_e] a_1,\dots,\alpha^{-1} [L\tilde K_e:\tilde K_e] a_r)\pmod{\calM_{L\tilde K_e}\Z_p^r})_{L\tilde K_e}\]
in $\varinjlim_L\Z_p^r/\calM_L\Z_p^r(\rhonr)$. Note that
\[\alpha^{-1}[L\tilde K_e:\tilde K_e]=\frac{[L\tilde K_e:\tilde K]}{\alpha[\tilde K_e:\tilde K]}=\frac{[L\tilde K_e:\tilde K]}{\alpha p^e}=\frac{[L\tilde K_e:\tilde K]}{[L:\tilde K]}=[L\tilde K_e:L]\]
and therefore
\[(([L\tilde K_e:L] a_1,\dots,[L\tilde K_e:L] a_r)\pmod{\calM_{L\tilde K_e}\Z_p^r})_{L\tilde K_e}\]
is in the image of $\Phi$. Again by Lemma \ref{sameasdegree}, this coincides with the element
\[((a_1,\dots,a_r)\pmod{\calM_L\Z_p^r})_L,\]
and this concludes the proof of surjectivity.

It remains to be proved that $\ker\Phi=\calM_{\tilde K}\Q_p^r$. Let us assume that $\left(\frac{a_1}{p^e},\dots,\frac{a_r}{p^e}\right)\in\ker\Phi$. Then $((a_1,\dots,a_r)\pmod{\calM_{\tilde K_e}\Z_p^r})_{\tilde K_e}=0$ in $\varinjlim_L\Z_p^r/\calM_L\Z_p^r(\rhonr)$. This means that there exists an extension $L/\tilde K_e$ such that
\[e_{L/\tilde K_e}\sum_{i=1}^{d_{L/\tilde K_e}-1}F_{\tilde K_e}^i\cdot (a_1,\dots,a_r)\equiv 0\pmod{\calM_L\Z_p^r}.\]
By the definition of $\calM_{\tilde K_e}$ we also know that
\[\begin{split}e_{L/\tilde K_e}\sum_{i=1}^{d_{L/\tilde K_e}-1}F_{\tilde K_e}^i\cdot (a_1,\dots,a_r)&\equiv e_{L/\tilde K_e}d_{L/\tilde K_e} (a_1,\dots,a_r) \pmod{\calM_{\tilde K_e}\Z_p^r}\\ &\equiv [L:\tilde K_e](a_1,\dots,a_r) \pmod{\calM_{\tilde K_e}\Z_p^r}.\end{split}\]
Since $\calM_L\Z_p^r\subseteq \calM_{\tilde K_e}\Z_p^r$, we deduce that $[L:\tilde K_e](a_1,\dots,a_r)\in \calM_{\tilde K_e}\Z_p^r$ and therefore $(a_1,\dots,a_r)\in \calM_{\tilde K_e}\Q_p^r=\calM_{\tilde K}\Q_p^r$. This proves the inclusion $\ker\Phi\subseteq \calM_{\tilde K}\Q_p^r$.

Conversely, let $\calM_{\tilde K}\left(\frac{a_1}{p^e},\dots,\frac{a_r}{p^e}\right)\in \calM_{\tilde K}\Q_p^r$, then
\[\Phi\left(\calM_{\tilde K}\left(\frac{a_1}{p^e},\dots,\frac{a_r}{p^e}\right)\right)=(\calM_{\tilde K}(a_1,\dots,a_r)\pmod{\calM_{\tilde K_e}\Z_p^r})_{\tilde K_e},\]

By the properties of $\tilde K$, we know that there exists a vector $v\in\Q_p^r$ such that
\[\calM_{\tilde K}(a_1,\dots,a_r)=\calM_{\tilde K_e}v.\]
It is also clear that there exists an $\tilde e\in\N$ such that $p^{\tilde e}v\in\Z_p^r$. Then we consider a totally ramified extension $\tilde L$ of $\tilde K_e$ of degree $p^{\tilde e}$. Because of total ramification, we have $\calM_{\tilde L}=\calM_{\tilde K_e}$, and hence we obtain
\[\begin{split}\Phi\left(\calM_{\tilde K}\left(\frac{a_1}{p^e},\dots,\frac{a_r}{p^e}\right)\right)&=(\calM_{\tilde K_e}(p^{\tilde e}v)\pmod{\calM_{\tilde L}\Z_p^r})_{\tilde L}\\
&=(\calM_{\tilde L}(p^{\tilde e}v)\pmod{\calM_{\tilde L}\Z_p^r})_{\tilde L}=0.\end{split}\]
Therefore we also have the inclusion $\calM_{\tilde K}\Q_p^r\subseteq \ker\Phi$.
\end{proof}

\begin{lemma}\label{FpKLurexact}
We have an exact sequence
\[0\rightarrow\varinjlim_L\Lambda_L^{G_L}\rightarrow\varinjlim_L\Lnr\xrightarrow{(F-1)\times 1}\varinjlim_L\Lnr \rightarrow \Q_p^r/\calM_{\tilde K}\Q_p^r(\rhonr)\rightarrow 0.\]
\end{lemma}

\begin{proof}
We use Lemma \ref{diagforlimit}, the fact that the direct limit functor is exact and Lemma \ref{QpMQp}.
\end{proof}

Our next step will be to apply the $G_N$-fixed point functor to a complex obtained from the above exact sequence in order to find an explicit representative of $R\Gamma\left(N, \Lambda_N^{G_N}\right)$.

\begin{lemma}\label{LurGcanNur}
We have
\[\left( \Mnr \right)^{\Gal(M/L)}=\iota_{M/L}\left( \Lnr \right).\]
In particular,
\begin{equation}\label{GN invariants}
\left( \varinjlim_L\Lnr \right)^{G_N} = \Nnr,
\end{equation}
where the direct limit is taken over all finite Galois extensions $L/K$. 
\end{lemma}

\begin{proof}
Up to some obvious modifications, the proof is the same as in \cite[Lemma 4.2.6]{BC2}.
\end{proof}

\begin{lemma}\label{acyclic}
The $G_K$-modules $\varinjlim_L \Lnr$ and $\Q_p^r/\calM_{\tilde K}\Q_p^r(\rhonr)$ are acyclic with respect to the fixed point functor $A \mapsto A^{G_N}$.
\end{lemma}

\begin{proof}
By \cite[Prop.~1.5.1]{NSW} we have for $i >0$
\[H^i(N, \varinjlim_L\Lnr) = \varinjlim_L H^i(\Gal(L/N), \Lnr)\]
which is trivial by Lemma \ref{LurInd}. The module $\Q_p^r/\calM_{\tilde K}\Q_p^r(\rhonr)$ is cohomologically trivial since it is uniquely divisible and the representation is continuous because $G_{\tilde K}$ acts trivially (see \cite[Prop.~1.6.2]{NSW}).
\end{proof}

\begin{lemma}\label{fixedpoints}
We have
\begin{equation}\label{QpMQpGN}(\Q_p^r/\calM_{\tilde K}\Q_p^r(\rhonr))^{G_N}=\Q_p^r/\calM_{N}\Q_p^r(\rhonr).\end{equation}
\end{lemma}

\begin{proof}
By the previous lemma the Tate cohomology of $\Q_p^r/\calM_{\tilde K}\Q_p^r(\rhonr)$ is trivial. We will use this to prove that the following sequence is exact:
\[0\to \calM_M\Q_p^r/\calM_{\tilde K}\Q_p^r(\rhonr)\to \Q_p^r/\calM_{\tilde K}\Q_p^r(\rhonr)\xrightarrow{T_{\tilde K/N}} (\Q_p^r/\calM_{\tilde K}\Q_p^r(\rhonr))^{G_N}\to 0.\]
The first map is injective since it is an inclusion.

Exactness in the middle term follows from $\hat H^{-1}(\Gal(\tilde K/N),\Q_p^r/\calM_{\tilde K}\Q_p^r(\rhonr))=0$.

From $\hat H^{0}(\Gal(\tilde K/N),\Q_p^r/\calM_{\tilde K}\Q_p^r(\rhonr))=0$ we deduce
\[(\Q_p^r/\calM_{\tilde K}\Q_p^r(\rhonr))^{G_N}=T_{\tilde K/N}(\Q_p^r/\calM_{\tilde K}\Q_p^r(\rhonr)),\]
i.e. the surjectivity of the trace map in the sequence. Therefore
\[(\Q_p^r/\calM_{\tilde K}\Q_p^r(\rhonr))^{G_N}\cong (\Q_p^r/\calM_{\tilde K}\Q_p^r)/(\calM_M\Q_p^r/\calM_{\tilde K}\Q_p^r)(\rhonr)\cong \Q_p^r/\calM_{N}\Q_p^r(\rhonr).\]

\end{proof}

We can now state and prove the following theorem. The proof works as in \cite[Thm.~4.2.8]{BC2}, but we prefer to rewrite it in detail for clarity.

\begin{thm}\label{fundamental1}
The complex
\[C^\bullet_{N, \rhonr} =\left[ \Nnr\xrightarrow{(F-1)\times 1}\Nnr\to\Q_p^r/\calM_{N}\Q_p^r(\rhonr)\right]\]
with non-trivial modules in degree $0$, $1$ and $2$ represents $R\Gamma\left(N, \Lambda_N^{G_N}\right)$.
\end{thm}

\begin{proof}
We write $X^\bullet$ for the standard resolution (as defined in 
\cite[Sec.~1.2]{NSW}) and $I^\bullet$ for an injective resolution of $\varinjlim_L\Lambda_L^{G_L}$. Let $C^\bullet$ be the complex 
\[\varinjlim_L\Lnr\xrightarrow{(F-1)\times 1}\varinjlim_L\Lnr \rightarrow\Q_p^r/\calM_{\tilde K}\Q_p^r(\rhonr).\]

The $G_K$-modules $\Map(G_K^n,\varinjlim_L\Lambda_L^{G_L})$ appearing in $X^\bullet$  are induced in the sense of \cite[Sec.~1.3]{NSW}. By Proposition 1.3.7 of loc.cit. the standard resolution consists therefore of acyclic modules with respect to the fixed point functor. By \cite[Th.~XX.6.2]{LangAlgebra} there exists a morphism of complexes $X^\bullet \lra I^\bullet$ inducing the identity on $\varinjlim_L\Lambda_L^{G_L}$ and isomorphisms on cohomology. By Lemma \ref{FpKLurexact} and \ref{acyclic} we may apply \cite[Th.~XX.6.2]{LangAlgebra} once again and obtain a morphism of complexes $C^\bullet \lra I^\bullet$, also inducing the identity on $\varinjlim_L\Lambda_L^{G_L}$ and isomorphisms on cohomology.
Therefore applying the $G_N$-fixed point functor together with (\ref{GN invariants}) and (\ref{QpMQpGN}) shows that $C_{N, \rhonr}^\bullet =\left( C^\bullet \right)^{G_N}$ is isomorphic to $R\Gamma\left(N, \Lambda_N^{G_N}\right)$ in the derived category.
\end{proof}

The following corollary is inspired by the analogous statement appearing in the first part of the proof of \cite[Thm.~4.20]{BreuPhd} and by \cite[Coroll.~4.2.10]{BC2}.

\begin{coroll}\label{Pbullet}
There exist quasi-isomorphisms of complexes
\[ \tilde P^\bullet := [P^{-1}\lra P^{0}\lra P^1] \lra \left[ \Nnr \xrightarrow{(F-1)\times 1} \Nnr \right],\]
and
\[P^\bullet := [P^{-1}\lra P^{0}\lra P^1\lra \Q_p^r/\calM_{N}\Q_p^r(\rhonr)]\lra C^\bullet_{N, \rhonr}\]
where the $\Z_p[G]$-modules $P^{-1}, P^{0}, P^1$ are finitely generated and projective and the module $\Q_p^r/\calM_{N}\Q_p^r(\rhonr)$ is $G$-cohomologically trivial.
\end{coroll}

\begin{proof}
By \cite[Prop.~XXI.1.1]{LangAlgebra} there exists a quasi-isomorphism
\[f\colon K^\bullet \lra \left[ \Nnr \xrightarrow{(F-1)\times 1} \Nnr \right],\]
where $K^\bullet := [A \lra B]$ is centered in degree $0$ and $1$ with $B$ a $\Z_p[G]$-projective module. Since by Theorem \ref{FpnZpomega} and Lemma \ref{UL1hat}, the complex 
\[ \left[ \Nnr \xrightarrow{(F-1)\times 1} \Nnr \right]\]
has finitely generated cohomology groups, the proof of \cite[Prop.~XXI.1.1]{LangAlgebra} actually shows that we can assume that $A$ and $B$ are both finitely generated.  By our Lemma \ref{LurInd} the module $\Nnr$ is cohomologically trivial. We hence  apply \cite[Prop.~XXI.1.2]{LangAlgebra} with $\mathfrak F$ being the family of cohomologically trivial modules. We obtain that $A$ is cohomologically trivial and hence has a two term resolution $0\lra P^{-1} \lra P^{0} \lra A\lra 0$ with finitely generated projective $\Z_p[G]$-modules $P^{-1}$ and $P^{0}$. It follows that there is a quasi-isomorphism 
\[ \tilde\eta:[P^{-1}\lra P^{0}\lra P^1] \to \left[ \Nnr \xrightarrow{(F-1)\times 1} \Nnr \right],\]
where $P^1 = B$. We obtain the following diagram of complexes
\[\xymatrix{P^{-1}\ar[r]^{\partial_P^{-1}}&P^0\ar[rr]^{\partial_P^{0}}\ar[d]^{\eta_0}&&P^1\ar[r]^-{w\circ\eta_1}\ar[d]^{\eta_1}&\Q_p^r/\calM_{N}\Q_p^r(\rhonr)\ar[d]^{=}\\
&\Nnr\ar[rr]^-{(F-1)\times 1}&&
\Nnr\ar[r]^-w&\Q_p^r/\calM_{N}\Q_p^r(\rhonr)}
\]
To conclude we need to show that the vertical maps form a quasi-isomorphism of complexes, i.e. that they induce isomorphisms on the cohomology groups. Actually it is enough to check this at the levels $1$ and $2$.

Since $\eta$ is a quasi-isomorphism, $\eta_1$ induces an isomorphism
\[P^1/\partial_P^0(P^0)\to \Nnr/((F-1)\times 1)\Nnr,\]
which restricts to an injective homomorphism
\[\ker(w\circ\eta_1)/\partial_P^0(P^0)\to \ker(w)/((F-1)\times 1)\Nnr.\]
To prove that it is also surjective, let us take $x\in\ker(w)$. Then there exist $a\in P^1$ and $b\in ((F-1)\times 1)\Nnr$ such that $x=\eta_1(a)+b$. We see now that $w(b)=0$ and by assumption $w(x)=0$, hence also $w(\eta_1(a))=0$, i.e. $a\in\ker (w\circ \eta_1)$. This proves also surjectivity.

It remains to prove that the morphism of complexes induces an isomorphism also on the second cohomology groups, or, equivalently, that $w(\Nnr)=w(\eta_1(P^1))$. We know that the map $w:\Nnr\to \Q_p^r/\calM_{N}\Q_p^r(\rhonr)$ passes to the quotient $w:\Nnr/((F-1)\times 1)\Nnr\to \Q_p^r/\calM_{N}\Q_p^r(\rhonr)$ and that $\eta_1$ induces an isomorphism $P^1/\partial_P^0(P^0)\to\Nnr/((F-1)\times 1)\Nnr$. Hence $w(\Nnr)=w(\eta_1(P^1))$.
\end{proof}

The following lemma is analogous to \cite[Lemma.~4.3.4]{BC2}, but we need to pay some attention in the proof, since in the present situation $P^\bullet$ does not comprise only projectives. We will focus on the peculiarity of the present setting without reproducing all the details of the proof which can be checked in \cite{BC2}.

\begin{lemma}\label{compatiblephi}
There exist quasi-isomorphisms of complexes
\[\varphi_n\colon P^\bullet/p^n\lra R\Gamma(N,\calF[p^n])[1]\]
which are compatible with the inverse systems. 
\end{lemma}

\begin{proof}
As in the first part of the proof of \cite[Lemma.~4.3.4]{BC2}, we can construct quasi-isomorphisms
\[\varphi_n:P^\bullet/p^n\to R\Gamma(N,\calF[p^{n}])[1]\]
such that
\[(\varphi_{n-1}\pi_n^P-\pi_n^Q\varphi_n)\pi_n=0\]
in $D(\Z_p[G])$, where
\[\pi_n:P^\bullet\to P^\bullet/p^n,\]
\[\pi_n^P:P^\bullet/p^n\to P^\bullet/p^{n-1}\]
and 
\[\pi_n^Q:R\Gamma(N,\calF[p^n])[1]\to R\Gamma(N,\calF[p^{n-1}])[1]\]
are the canonical projections. Let us consider the complex
\[\tilde P^\bullet=[P^{-1}\to P^0\to P^1].\]
Note that there is a natural inclusion
\[\iota:\tilde P^\bullet\to P^\bullet\]
and, since $P^\bullet/p^n=\tilde P^\bullet/p^n$, also a natural projection
\[\tilde\pi_n:\tilde P^\bullet\to P^\bullet/p^n\]
such that the diagram
\[\xymatrix{\tilde P^\bullet\ar[d]^-{\iota}\ar[dr]^-{\tilde\pi_n}\\
P^\bullet\ar[r]^-{\pi_n}&P^\bullet/p^n}\]
commutes. We deduce that
\[(\varphi_{n-1}\pi_n^P-\pi_n^Q\varphi_n)\pi_n\iota=(\varphi_{n-1}\pi_n^P-\pi_n^Q\varphi_n)\tilde\pi_n=0\]
in $D(\Z_p[G])$. Since $\tilde P^\bullet$ is bounded from above and comprising only projectives, we can apply \cite[Lemma~4.3.3]{BC2} and deduce that $(\varphi_{n-1}\pi_n^P-\pi_n^Q\varphi_n)\tilde\pi_n$ is homotopic to zero. As in the proof of \cite[Lemma.~4.3.4]{BC2}, we can now use the surjectivity of $\tilde\pi_n$ to deduce that also $\varphi_{n-1}\pi_n^P-\pi_n^Q\varphi_n$ is homotopic to zero. At this point we are again in the same situation of \cite[page~1367]{BurnsFlach98}, as it was the case in the proof of \cite[Lemma.~4.3.4]{BC2}, and we can easily conclude.
\end{proof}

We are now ready to prove Theorem \ref{mainthm}, of which we recall the formulation.

\begin{thm}\label{fundamental}
The complex
\[C^\bullet_{N,T} := \left[ \Nnr \xrightarrow{(F-1)\times 1} \Nnr \right]\]
with non-trivial modules in degree $1$ and $2$ represents $R\Gamma(N,T)$.
\end{thm}

\begin{proof}
The proof is the same as that of \cite[Thm.~4.3.1]{BC2}, using the corresponding results we have obtained in this section. 
\end{proof}

We can easily deduce the following corollary.

\begin{coroll}\label{Hi}
With the above notation we have
\begin{enumerate}
\item[(i)] $H^1(N, T) \cong \Lambda_N^{G_N}\cong \calF(\p_N)\times \zz^{G_N}$,
\item[(ii)] $H^2(N, T) \cong \zz/ (F_N-1)\zz$,
\item[(iii)] $H^i(N, T) = 0 \text{ for } i \ne 1,2$.
\end{enumerate}
\end{coroll}

\begin{proof}
This is straightforward by Theorem \ref{FpnZpomega}, Theorem \ref{fundamental} and Lemma \ref{UL1hat}.
\end{proof}

\begin{example}
The special case in which the representation $\rhonr$ is totally decomposable into a product of characters $\chinr_i$ is very easy, already using the results of \cite{BC2}. To that aim it is enough to notice that also $\Nnr$ decomposes into a direct product of the corresponding $\mathcal I_{N/K}(\chinr)$ defined in \cite{BC2} and the same holds for the modules appearing in $R\Gamma(N,T)$. Therefore Theorem \ref{fundamental} is in this case an easy consequence of \cite[Thm.~2]{BC2}.
\end{example}

\begin{example}
Let us consider a different example, which can not be handled with the results of \cite{BC2}. Let $\rhonr = \rho_u \colon G_{\Q_p} \lra \Gl_r(\Z_p)$ be the unramified representation attached to
\[u=\begin{pmatrix}1&1\\0&1\end{pmatrix}\]
by Proposition \ref{unram reps prop}. Then
\[\calM_N=u^{d_N}-I=\begin{pmatrix}0&d_N\\0&0\end{pmatrix},\]
where $d_N$ is the inertia degree of $N/\Q_p$. By Corollary \ref{Hi} and recalling Lemma \ref{UL1hat}, we obtain
\begin{enumerate}
\item[(i)] $H^1(N, T) \cong \calF(\p_N)\times \Z_p$,
\item[(ii)] $H^2(N, T) \cong \Z_p/d_N\Z_p\times \Z_p$.
\end{enumerate}
\end{example}

\end{section}

\begin{section}*{Acknowledgements}
I would like to thank Werner Bley for encouraging me to generalize some of the results from our common paper \cite{BC2}, for sharing some of his insights and for his careful reading of the final version of the manuscript and for suggesting some improvements, including a new and nicer version of Lemma \ref{fixedpoints}. I am also very grateful to Cornelius Greither for his comments and for helping me to fix some problems.
\end{section}

\end{document}